\documentclass[reqno]{amsart}

\usepackage{pstricks}
\usepackage{multido}

\usepackage{graphicx}

\usepackage{amssymb}
\usepackage{amsmath}
\usepackage{amstext}
\usepackage{amsbsy}
\usepackage{amsxtra}
\usepackage{amsfonts}
\usepackage{latexsym}
\usepackage{alltt}
\usepackage{bbm}

\usepackage{verbatim}
\usepackage{amscd}
\usepackage[all]{xy}

\theoremstyle{plain}
\newtheorem{theorem}{Theorem}[section]
\newtheorem{corollary}[theorem]{Corollary}
\newtheorem{lemma}[theorem]{Lemma}
\newtheorem{proposition}[theorem]{Proposition}

\newcommand{\C}{\mathbb C}
\newcommand{\R}{\mathbb R}

\newcommand{\Z}{\mathbb Z}

\newcommand{\Sl}{\mathfrak{sl}(2,\C)}
\newcommand{\SU}{\mathrm{SU}(2)}
\newcommand{\su}{\mathfrak{su}(2)}

\numberwithin{equation}{section}

\title[Infinitesimal deformations of cmc surfaces of finite type in $\mathbb{S}^3$]{On infinitesimal deformations of cmc surfaces of finite type in the 3-sphere}

\author{M. Kilian}
\address{M. Kilian, Department of Mathematics,
University College Cork, Ireland.}
\email{m.kilian@ucc.ie}

\author{M. U. Schmidt}
\address{Institut f\"ur Mathematik,
Universit\"at Mannheim, 68131 Mannheim, Germany}
\email{schmidt@math.uni-mannheim.de}

\begin{document}

\thanks{{\it Mathematics Subject Classification.}
53A10, 53C17. \today}

%=========================

\begin{abstract}
We describe infinitesimal deformations of constant mean curvature
surfaces of finite type in the 3-sphere. We use Baker-Akhiezer
functions to describe such deformations, as well as polynomial
Killing fields and the corresponding spectral curve to distinguish
between isospectral and non-isospectral deformations.
\end{abstract}
\maketitle

%%%%%%%%%%%%%%%%%%%%%%%%%%%%%%%%%%%%%

\begin{center}
  \sc Introduction
\end{center}

The theory of constant mean curvature ({\sc{cmc}}) surfaces, and
more generally that of harmonic maps has developed greatly over the
past decades. One reason for this may be that there are two main
approaches possible towards the subject which cross fertilize each
other: geometric PDE methods and integrable system techniques. In
the late 1990's the integrable system approach culminated in a very
general (local) description of such harmonic maps in terms of loop
Lie algebra valued 1-forms by Dorfmeister, Pedit and Wu. In
particular, in the case of {\sc{cmc}} tori, Pinkall and Sterling
\cite{PinS}, and independently Hitchin \cite{Hit:tor} showed that a
solution to the structure equation (the sinh-Gordon equation) can be
represented by a hyperelliptic Riemann surface, the so called
spectral curve. The solution in this case is said to be of finite
type.

Pinkall and Sterling \cite{PinS} construct infinitesimal
deformations of {\sc{cmc}} tori in $\mathbb{R}^3$. We carry over
their constructions to {\sc{cmc}} tori in $\mathbb{S}^3$. Such tori
in $\mathbb{S}^3$ have non-isospectral deformations changing the
mean curvature in contrast to {\sc{cmc}} tori in $\mathbb{R}^3$. For
this reason we consider also deformations that change the mean
curvature. The corresponding normal variation then obeys an
inhomogenuos Jacobi equation.

We briefly outline the contents of the paper: In the first section
we recall some facts about {\sc{cmc}} surfaces in $\mathbb{S}^3$ and
introduce our notation. In particular, we recall the notion of
extended frame and spectral curve for {\sc{cmc}} surfaces of finite
type. In the second section we construct Jacobi fields and
parametric Jacobi fields for {\sc{cmc}} tori in $\mathbb{S}^3$. In
the third section we construct homogenous Jacobi fields in terms of
the Baker-Akhiezer function of the $\sinh$-Gordon equation. Finally
we show that the Fermi curve of the Jacobi operator is isomorphic to
the spectral curve of the $\sinh$-Gordon equation. In the last
section we extend the deformation to a deformation of the
corresponding polynomial Killing field, and exhibit both
isospectral, and non-isospectral deformations in terms of polynomial
Killing fields.

\section{Conformal {\sc{cmc}} immersions into $\mathbb{S}^3$}
\label{sec:CONF}

\subsection{Extended frames}
We identify the 3-sphere $\mathbb{S}^3 \subset \R^4$ with
$\mathbb{S}^3 \cong \left(\,\SU\times\SU\,\right)/\,\mathrm{D}$,
where $\mathrm{D}$ is the diagonal in $\SU\times\SU$.
The Lie algebra of the matrix Lie group $\SU$ is
$\mathfrak{su}(2)$, equipped with the commutator
$[\,\cdot,\,\cdot\,]$.
For $\alpha,\,\beta \in \Omega^1(T\R^2,\mathfrak{su}(2))$
smooth $1$--forms on $\R^2 \cong T\R^2$ with values in
$\mathfrak{su}(2)$,
we define the $\mathfrak{su}(2)$--valued
$2$--form
\begin{equation*}
  [ \alpha \wedge \beta ] (X,\,Y) =
  [\alpha(X),\,\beta(Y)] - [\alpha(Y),\,\beta(X)],
\end{equation*}
for $X,\,Y \in T\R^2$. Let $L_g:h \mapsto gh$ be left
multiplication in $\SU$. Then by left translation,
the tangent bundle is $T\SU \cong \SU \times \mathfrak{su}(2)$
and $\theta : T\SU \to \mathfrak{su}(2),\,
v_g \mapsto (dL_{g^{-1}})_g v_g$ is
the (left) Maurer--Cartan form. It satisfies the
Maurer--Cartan equation
\begin{equation} \label{eq:MC_equation}
 2\, d\theta + [ \theta \wedge \theta ] = 0.
\end{equation}
For a map $F:\R^2 \to \SU$, the pullback
$\alpha = F^*\theta$ also satisfies \eqref{eq:MC_equation},
and conversely, every
solution $\alpha \in \Omega^1(\R^2,\mathfrak{su}(2))$ of
\eqref{eq:MC_equation} integrates to a smooth map
$F:\R^2 \to \SU$ with $\alpha = F^*\theta$.

Complexifying the tangent bundle $T\R^2$ and decomposing
into $(1,0)$ and $(0,1)$ tangent spaces, and writing
$d=\partial + \bar{\partial}$, we may split
$\omega \in \Omega^1(M,\mathfrak{su}(2))$ into the
$(1,0)$ part $\omega^\prime$, the $(0,1)$ part
$\omega^{\prime \prime}$ and write
$\omega = \omega^\prime + \omega^{\prime\prime}$.
We set the $*$--operator on $\Omega^1(M,\Sl)$ to
$*\omega = - i \omega^\prime + i \omega^{\prime \prime}$.
Fix $\epsilon \in \su$ and let
$\mathbb{T} = \mathrm{stab}(\epsilon)$ be the stabilizer
of $\epsilon$ under the adjoint action of $\SU$ on $\su$.
We shall view the 2-sphere $\mathbb{S}^2$ as the quotient
$\mathbb{S}^2 \cong \SU / \mathbb{T}$.

We denote by $\langle \cdot \, , \cdot \rangle$
the bilinear extension of the Ad--invariant inner
product of $\su$ to $\su^{\mbox{\tiny{$\C$}}} = \Sl$ such that
$\langle \epsilon,\,\epsilon \rangle = 1$.
The double cover of the isometry group
$\mathrm{SO}(4)$ is $\SU \times \SU$ via the action
$X \mapsto FXG^{-1}$.
Writing $df = df^\prime + df^{\prime\prime}$ for the
differential of $f$, recall that an immersion
$f:\R^2 \to \SU$ is conformal if and only if
\begin{equation} \label{eq:conformal}
  \langle df^\prime ,\, df^\prime \rangle = 0.
\end{equation}
If $f:\R^2 \to \mathbb{S}^3$ is a conformal immersion and
$\omega = f^{-1}df$, then it can be shown (see e.g \cite{SKKR})
that the mean curvature function $H$ of $f$ is given by
\begin{equation} \label{eq:H_S3}
  2\,d*\omega = H \,[ \omega \wedge \omega ].
\end{equation}
Suppose $f:\R^2 \to \SU$ is a conformal immersion with
non-zero constant mean curvature $H$.
Then $2d*\omega = H\,[\,\omega \wedge \omega \,]$ and
$2d\omega + [\,\omega \wedge \omega \,]=0$ combined give
$d*\omega +H^{-1}d*\omega =0$, or alternatively
\begin{equation} \label{eq:w'}
  (1-iH^{-1})d\omega' + (1+iH^{-1})d\omega'' = 0.
\end{equation}
Inserting
$2d\omega'' = -2d\omega' - [\,\omega \wedge \omega \,]$
respectively
$2d\omega' = -2d\omega'' - [\,\omega \wedge \omega \,]$
into \eqref{eq:w'} gives
$4d\omega' = (iH-1)[\,\omega \wedge \omega\,]$ and
$4d\omega'' = -(1+iH)[\,\omega \wedge \omega\,]$.
Then
\begin{equation*}
  \alpha_\lambda = \tfrac{1}{2}(1-\lambda^{-1})(1+iH)\,\omega'
  + \tfrac{1}{2}(1-\lambda)(1-iH)\,\omega''
\end{equation*}
satisfies $2d\alpha_\lambda + [\,\alpha_\lambda \wedge
\alpha_\lambda \,] = 0$ for all $\lambda \in \C^*$, and thus there
exists a corresponding \emph{extended frame} $F_\lambda : \R^2
\times \mathbb{S}^1 \to \SU$ with $dF_\lambda =
F_\lambda\,\alpha_\lambda$ and $F_\lambda(0) = \mathbbm{1}$.
Conversely, we recall the following version of a result by Bobenko
\cite{Bob:cmc}. Our formulas are slightly different to those of
Bobenko \cite{Bob:tor,Bob:cmc,Bob:2x2}, but the modifications in the
proof are obvious.
\begin{theorem} \label{th:sinh}
Let $u:\R^2 \to \R$ be a smooth function and define
\begin{equation} \label{eq:general_alpha}
  \alpha_\lambda = \frac{1}{2}\,\begin{pmatrix}
  u_z\,dz-u_{\bar{z}}\,d\bar{z} &
  i\,\lambda^{-1}e^u\,dz + i\,e^{-u}\,d\bar{z}\\
  i\,e^{-u}\,dz + i\,\lambda\,e^{u}\,d\bar{z}&
  -u_z\,dz+u_{\bar{z}}\,d\bar{z}
  \end{pmatrix}\,.
\end{equation}
Then $2\,d\alpha_\lambda +
[\,\alpha_\lambda \wedge \alpha_\lambda\,] = 0$
if and only if $u$ is a solution of the sinh-Gordon equation
\begin{equation}\label{eq:sinh-Gordon}
  \partial \bar{\partial}u + \tfrac{1}{2}\,\sinh (2u) = 0.
\end{equation}
Furthermore, for any solution $u$ of the sinh-Gordon equation
and corresponding extended frame
$F_\lambda:\R^2 \times \mathbb{S}^1 \to \SU$, and
$\lambda_0,\,\lambda_1 \in \mathbb{S}^1,\,
\lambda_0 \neq \lambda_1$, the map $f:\R^2 \to \SU$
defined by
\begin{equation}\label{eq:Sym_S3}
  f = F_{\lambda_1} F_{\lambda_0}^{-1}
\end{equation}
is a conformal immersion with constant mean curvature
\begin{equation} \label{eq:H_mu}
  H = i \,\frac{\lambda_0 + \lambda_1}{\lambda_0 -\lambda_1},
\end{equation}
conformal factor
\begin{equation} \label{eq:v^2}
  v^2 = \frac{e^{2u}}{H^2 + 1}\,,
\end{equation}
and Hopf differential $Q\,dz^2$ with
\begin{equation} \label{eq:Q}
  Q = \frac{i}{4}\,(\lambda_1^{-1}-\lambda_0^{-1}).
\end{equation}
\end{theorem}
\begin{corollary} Useful formulae obtained from equations
\eqref{eq:H_mu}, \eqref{eq:v^2} and \eqref{eq:Q} are
\begin{equation} \label{eq:useful}
  v^2 = 4e^{2u} Q\,\overline{Q}\,, \qquad
  4 Q\,\overline{Q} (H^2 +1) = 1.
\end{equation}
The constant mean curvature surface $f$ in $\mathbb{S}^3$
corresponding to a solution $u$ of the sinh-Gordon equation
satisfies the following equations with respect to the
moving frame $(f,\,\partial f,\,\bar{\partial}f,\,N)$:
\begin{equation} \begin{split}\label{eq:Frenet}
    \partial N &= -H\,\partial f + 2\,v^{-2}Q\,\bar{\partial}f \\
    \partial \partial f &= 2u_z\,\partial f - Q\,N \\
    2\,\partial \bar{\partial} f &= -v^2 \,f + v^2 H\,N
    \end{split}
\end{equation}
\end{corollary}
\subsection{Spectral curve}
Let $f:\R^2 \to \SU$ be a conformal {\sc{cmc}} immersion with
corresponding extended frame $F_\lambda$. For a translation $\tau:
\R^2 \to \R^2$ we write $\tau^* f = f \circ \tau$, and define the
monodromy of $F_\lambda$ with respect to $\tau$ as
\begin{equation*} %\label{eq:Monodromy}
  M(\tau,\,\lambda) = \tau^*F_\lambda \, F^{-1}_\lambda\,.
\end{equation*}
Now suppose we have a conformal {\sc{cmc}} immersion of a torus
$f:\R^2 / \Gamma \to \SU$, with lattice
\begin{equation*} %\label{eq:lattice}
  \Gamma = \omega_1\,\Z \,\oplus \,\omega_2\,\Z\,,
\end{equation*}
and corresponding extended frame $F_\lambda$. If
$M_1(\lambda),\,M_2(\lambda)$ are the monodromies of $F_\lambda$
with respect to $\omega_1$ and $\omega_2$, and $\mu_1,\,\mu_2$ the
corresponding eigenvalues, then the \emph{spectral curve} of the
torus is
\begin{equation*}
  \Sigma_f = \{ (\lambda,\,\mu_1,\,\mu_2) :
  \det(\mu_1\,\mathbbm{1} - M_1(\lambda)) =
  \det(\mu_2\,\mathbbm{1} - M_2(\lambda)) = 0 \,\}.
\end{equation*}

We recall the description of {\sc{cmc}} tori in terms of spectral
curves:

Let $Y$ be a hyperelliptic Riemann surface with branch points over
$\lambda = 0 \,(y^+)$ and $\lambda = \infty \,(y^-)$. Then $Y$ is
the spectral curve of an immersed {\sc{cmc}} torus in the three
sphere if and only if the following four conditions hold:
\begin{enumerate}
\item Besides the hyperelliptic involution $\sigma$, the surface
$Y$ has two further anti-holomorphic involutions $\eta$ and $\varrho
= \eta \circ \sigma = \sigma \circ \eta$, such that $\eta$ has no
fixed points and $\eta(y^+) = y^-$.
\item There exist two non-zero holomorphic functions
$\mu_1,\,\mu_2$ on $Y\setminus\{y^+,\,y^-\}$ such that for $i=1,\,2$
\begin{equation*}
  \sigma^*\mu_i = \mu_i^{-1},\,\eta^*\bar{\mu}_i = \mu_i,\,
  \varrho^*\bar{\mu}_i = \mu_i^{-1}.
\end{equation*}
\item The forms $d \ln \mu_i$ are meromorphic differentials
of the second kind with double poles at $y^\pm$. The singular parts
at $y^+$ respectively $y^-$ of these two differentials are linearly
independent.
\item There are four fix points $y_1,\,y_2 =
\sigma(y_1),\,y_3,\,y_4 = \sigma(y_3)$ of $\varrho$, such that the
functions $\mu_1$ and $\mu_2$ are either $1$ or $-1$ there.
\end{enumerate}
%
%%%
%
We shall describe spectral curves of {\sc{cmc}} tori in
$\mathbb{S}^3$ via hyperelliptic surfaces of the form
\begin{equation*}
  \nu^2 = \left\{ \begin{array}{ll} \lambda\,a(\lambda)
      &\mbox{ if $g$ is even,} \\ a(\lambda)
      &\mbox{ if $g$ is odd.}\end{array} \right.
\end{equation*}
Here $\lambda:Y \to \C\mathbb{P}^1$ is chosen so that $y^\pm$
correspond to $\lambda=0$ respectively $\lambda=\infty$, and
\begin{equation*}
  \sigma^*\lambda = \lambda,\,
  \eta^*\bar{\lambda} = \lambda^{-1},\,
  \varrho^*\bar{\lambda} = \lambda^{-1},
\end{equation*}
and
\begin{equation*}
  a(\lambda) = \prod_{i=1}^g (\lambda - \alpha_i)
  (\lambda^{-1} - \bar{\alpha}_i)
\end{equation*}
with pair-wise different branch points $\alpha_1,\,\ldots,\,\alpha_g
\in \{ \lambda \in \C : |\lambda| <1\}$. Hence $\eta^*\bar{a} = a$
and $\varrho^*\bar{a} = a$. For $|\lambda|=1$ we have that
$a(\lambda) >0$, and since $\eta$ has no fix points, we have
\begin{equation*}
  \left\{ \begin{array}{lllll}
    &\eta^*\bar{\nu} = -\lambda^{-1}\nu,
    &\varrho^*\bar{\nu} = \lambda^{-1}\nu,
    &\sigma^*\nu = -\nu
    &\mbox{ if $g$ is even,} \\
      &\eta^*\bar{\nu} = -\nu,
    &\varrho^*\bar{\nu} = \nu,
    &\sigma^*\nu = -\nu
    &\mbox{ if $g$ is odd.} \end{array} \right.
\end{equation*}
Up to now, the parameter $\lambda$ is only determined up to a
rotation. Pick $\lambda_0,\,\lambda_1 \in \mathbb{S}^1,\, \lambda_1
\neq 1$, and take the unique parameter $\lambda$ for which the
points $y_1$ and $y_2 = \sigma(y_1)$ correspond to the the two
points over $\lambda =\lambda_0$. Then the points $y_3$ and $y_4 =
\sigma(y_3)$ correspond to the two points over $\lambda =
\lambda_1$.

%%%%%%%%%%%%%%%%%%%%%%%%%%%%%%%%%%%%%%%%%%%%%%%%%%%%%%%%%%%%%%
%%%%%%%%%%%%%%%%%%%%%%%%%%%%%%%%%%%%%%%%%%%%%%%%%%%%%%%%%%%%%%
%%%%%%%%%%%%%%%%%%%%%%%%%%%%%%%%%%%%%%%%%%%%%%%%%%%%%%%%%%%%%%
%
\section{Jacobi fields}
An infinitesimal deformation of a {\sc{cmc}} surface $f$ by
{\sc{cmc}} surfaces is given by a normal vector field $\dot{f} =
\omega\,N$, where the function $\omega:\mathbb{R}^2 \to \mathbb{R}$
has to be a solution of a Jacobi equation. For such a deformation in
Euclidean 3-space, it turns out that $\omega$ and the infinitesimal
change of the conformal factor solve the same homogeneous Jacobi
equation. We shall see that this is not the case for {\sc{cmc}}
surfaces in the 3-sphere. The reason for this is that the mean
curvature changes throughout a deformation in a non-trivial way,
introducing an inhomogeneity in the Jacobi equation for $\omega$.

Assume we are given  a smooth one-parameter family $f_t:\R^2 \to
\SU,\,t \in (-\epsilon,\,\epsilon)$ of {\sc{cmc}} surfaces of finite
type, with appropriate solutions $u_t$ of the sinh-Gordon equation,
mean curvatures $H_t$, and Hopf differentials $Q_t\,dz^2$, and
normals $N_t$. Denoting differentiation with respect to $t$ by a
dot, and omitting the subscript $t$ when $t=0$, we then have
\begin{equation} \label{eq:Jacobifield}
  \dot{f} = \tau\,\partial f +
  \sigma\,\bar{\partial}f + \omega\,N
\end{equation}
with smooth functions $\tau,\,\sigma :\C  \to \C$ such that
$\bar{\tau} = \sigma$, and a smooth funtion $\omega:\C \to \R$.

The following is the analogue of Proposition 2.1 in Pinkall
and Sterling \cite{PinS}.
\begin{proposition} \label{th:PinS1}
Every Jacobi field $\omega\,N$ along $f$ can be supplemented
by a tangential component
$\tau\,\partial f + \sigma\,\bar{\partial}f$
to yield a parametric Jacobi field.
Further, if \eqref{eq:Jacobifield} is a parametric Jacobi field,
then $\omega$ solves the inhomogeneous Jacobi equation
\begin{equation} \label{eq:inhom_jacobi}
  \bar{\partial}\partial\omega  + \cosh(2u)\,\omega =
  \frac{\dot{H}\,e^{2u}}{2(H^2+1)}\,,
\end{equation}
while $\dot{u}$ solves the homogeneous Jacobi equation
\begin{equation} \label{eq:dot_u}
  \bar{\partial}\partial\dot{u} + \cosh(2u)\,\dot{u}\, = 0.
\end{equation}
\end{proposition}
\begin{proof}
Straightforward computations using \eqref{eq:Frenet} give
\begin{equation*} \begin{split}
    \partial \dot{f} &= -\frac{\sigma v^2}{2}\,f +
    (\partial \tau + 2\tau \partial u-\omega H)\,\partial f +
    (\partial \sigma + \frac{2\omega Q}{v^2})\,\bar{\partial}f +
    (\frac{\sigma v^2}{2}\,H - \tau Q + \partial \omega)\,N\,, \\
    \bar{\partial} \dot{f} &= -\frac{\tau v^2}{2}\,f +
    (\bar{\partial} \tau + \frac{2\omega \bar{Q}}{v^2})\,\partial f +
      (\bar{\partial} \sigma + 2\sigma \bar{\partial}u -
    \omega H)\,\bar{\partial}f +
    (\frac{\tau v^2}{2}\,H - \sigma \bar{Q} +
    \bar{\partial}\omega )\,N\,.
\end{split}
\end{equation*}
Conformality $\langle \partial f,\,\partial f \rangle =
\langle \bar{\partial}f,\,\bar{\partial}f \rangle = 0$
implies that
$\langle \partial \dot{f},\,\partial f \rangle =
\langle \bar{\partial} \dot{f},\,\bar{\partial}f \rangle = 0$,
and consequently for $\tau,\,\sigma$ to supplement $\omega N$
to a parametric Jacobi field they must satisfy
\begin{equation} \begin{split} \label{eq:sigma_z}
  \partial \sigma &= -2\,v^{-2}\omega\, Q\,, \\
  \bar{\partial} \tau &= -2\,v^{-2}\omega\,\overline{Q}\,.
\end{split}
\end{equation}
Hence the above expressions for $\partial\dot{f}$ and
$\bar{\partial}\dot{f}$ simplify to
\begin{equation} \begin{split} \label{eq:fdot}
    \partial \dot{f} &= -\frac{\sigma v^2}{2}\,f +
    (\partial \tau + 2\tau \partial u-\omega H)\,\partial f +
    (\frac{\sigma v^2}{2}\,H - \tau Q + \partial \omega)\,N\,, \\
    \bar{\partial} \dot{f} &= -\frac{\tau v^2}{2}\,f +
     (\bar{\partial} \sigma + 2\sigma \bar{\partial}u -
    \omega H)\,\bar{\partial}f +
    (\frac{\tau v^2}{2}\,H - \sigma \overline{Q} +
    \bar{\partial}\omega )\,N\,.
\end{split}
\end{equation}
Differentiating $v^2 = 2\langle \partial f,\,
\bar{\partial}f \rangle$ gives
\begin{equation*}
    v\dot{v} = \langle \partial \dot{f},\,
    \bar{\partial}f \rangle +
    \langle \partial f,\,\bar{\partial} \dot{f} \rangle =
    \tfrac{1}{2}v^2
    (\partial \tau + 2\tau \partial u + \bar{\partial} \sigma +
    2\sigma \bar{\partial}u - 2\omega H ),
\end{equation*}
and combining this with the derivative of
equation \eqref{eq:v^2} yields
\begin{equation} \label{eq:udot}
  \dot{u} =
  \frac{1}{2}\partial \tau + \tau \partial u +
  \frac{1}{2}\bar{\partial} \sigma +
    \sigma \bar{\partial}u - \omega H\, +
    \frac{H\dot{H}}{H^2 +1}.
\end{equation}
From \eqref{eq:Frenet} we have
$2\,\langle \bar{\partial}\partial f ,\,N \rangle = v^2\,H$,
and differentiating this gives on the one hand
\begin{equation*}
  \langle \bar{\partial}\partial\dot{f},\,N \rangle +
  \langle \bar{\partial}\partial f ,\,\dot{N} \rangle =
  v\dot{v}H + \tfrac{v^2}{2}\dot{H}.
\end{equation*}
Differentiating
$2\,\bar{\partial}\partial f  = -v^2\,f + v^2H\,N$
gives
\begin{equation*}
  2\,\bar{\partial}\partial\dot{f} = -2v\dot{v}\,f -
  v^2\,\dot{f} +
  2\,v\dot{v}H\,N + v^2\dot{H}\,N + v^2 H\,\dot{N}\,.
\end{equation*}
Hence $2\langle \bar{\partial}\partial\dot{f},\,N \rangle =
- v^2\omega + 2v\dot{v}H + v^2\dot{H}$, and consequently
\begin{equation*}
  2\,\langle \bar{\partial}\partial f ,\,\dot{N} \rangle
  = v^2 \omega.
\end{equation*}
On the other hand, differentiating the first equation of
\eqref{eq:fdot} with respect to $\bar{z}$, and then computing
$\langle \bar{\partial}\partial\dot{f},\,N \rangle$ gives
\begin{equation*} \begin{split}
  \langle \bar{\partial}\partial\dot{f},\,N \rangle =
  &\tfrac{1}{2}(\partial \tau + 2\tau \partial u-\omega H)\,
  v^2 H -
   \overline{Q}(\partial \sigma + 2\omega v^{-2}Q) \\
    &+ \tfrac{1}{2}v^2 H \bar{\partial} \sigma  +
    \sigma v\bar{\partial}v H
    -\bar{\partial} \tau \,Q + \bar{\partial}\partial\omega ,
    \end{split}
\end{equation*}
and equating this with the first expression we obtained
for $2\,\langle \bar{\partial}\partial\dot{f},\,N \rangle$,
we obtain \eqref{eq:inhom_jacobi}.

Equation \eqref{eq:dot_u}
is now proven by a direct calculation using
equations \eqref{eq:udot}, \eqref{eq:inhom_jacobi} and the
time derivative of the second equation in \eqref{eq:useful}.

It remains to prove that the differential equations
defining $\sigma,\,\tau$ are integrable. From
\begin{equation*} \begin{split}
  &\langle \dot{N},\,f \rangle = -\langle N,\,\dot{f} \rangle
  = - \omega\,, \\
  &\langle \dot{N},\,\partial f \rangle =
  -\langle N,\,\partial \dot{f} \rangle
  = \tau Q - \sigma \tfrac{v^2}{2}H - \partial \omega\,, \\
  &\langle \dot{N},\,\bar{\partial}f \rangle =
    -\langle N,\,\bar{\partial} \dot{f} \rangle
    = \sigma Q - \tau \tfrac{v^2}{2}H - \bar{\partial}\omega\,,
\end{split}
\end{equation*}
we compute
\begin{equation} \label{eq:dotN}
  \dot{N} = -\omega\,f + v^{-2}(
  2\sigma Q - \tau v^2 H - 2\bar{\partial}\omega )\,\partial f
  + v^{-2}(2\tau Q - \sigma v^2H -
  2\partial \omega)\,\bar{\partial}f\,.
\end{equation}
Differentiating $Q = -\langle \partial^2 f ,\,N \rangle$ we have
$\dot{Q} = -\langle \partial^2 \dot{f},\,N \rangle -
\langle \partial^2 f ,\,\dot{N} \rangle$,
and with \eqref{eq:dotN} learn that
\begin{equation*}
  \langle \partial^2 f ,\,\dot{N} \rangle =
  ( 2\tau Q - \sigma v^2H - 2\partial \omega)\,\partial u.
\end{equation*}
Differentiating the first equation of \eqref{eq:fdot}
with respect to $z$ and taking inner product against $N$ gives
\begin{equation*}
  \langle \partial^2 \dot{f},\,N \rangle = \omega QH -
  2Q\partial \tau-2Q\tau \partial u
  + \tfrac{1}{2}v^2 H\partial \sigma +
  \sigma v H \partial v  + \partial^2 \omega ,
\end{equation*}
and therefore $\dot{Q} = 2Q\partial \tau + 2\partial u \,\partial
\omega - \partial^2 \omega$. Solving this equation for $\partial
\tau$ gives
\begin{equation} \label{eq:tau_z}
    \partial \tau =
    \frac{\dot{Q}+\partial^2 \omega -2\partial u\partial \omega}
     {2Q}\,.
\end{equation}
Finally, a straightforward computation proves that
integrability for $\tau$ holds automatically, that is
differentiating $\bar{\partial} \tau$ in \eqref{eq:sigma_z}
with respect to $z$ gives the same as differentiating
$\partial \tau$ in \eqref{eq:tau_z} with respect to $\bar{z}$.
Also, the analogous statement then holds for $\sigma$, that is
differentiating
\begin{equation} \label{eq:sigma_zbar}
    \bar{\partial} \sigma =
    \frac{\dot{\overline{Q}} + \bar{\partial}^2 \omega -
      2\bar{\partial} u \bar{\partial} \omega}
     {2\overline{Q}}\,.
\end{equation}
with respect to $z$ gives the same as
differentiating $\partial \sigma$ in \eqref{eq:sigma_z}
with respect to $\bar{z}$. Hence the functions $\tau,\,\sigma$ that supplement
the Jacobi field are unique up to addition of complex constants.
We compute $\tau,\,\sigma$ explicitly in terms of Baker-Akhiezer
functions in \eqref{eq:tau&sigma}.
\end{proof}
%
%%%%%%%%%%%%%%%%%%%%%%%%%%%%%%%%%%%%%%%%%%%%%%%%%%%%%%%%%%
%
The following is the analogue of Proposition 2.2 in Pinkall and
Sterling \cite{PinS}. We call a parametric Jacobi field a Killing
field, if it is induced by an infinitesimal isometry of
$\mathbb{S}^3$.
\begin{proposition} \label{th:PinS2}
A parametric Jacobi field is a Killing
field if and only if $\dot{u}=0$.
\end{proposition}
\begin{proof}
The 'only if' part is trivial. A parametric Jacobi field \eqref{eq:Jacobifield}
is a Killing field, if and only if there exists elements
$g_0,g_1\in\su$, such that
\begin{equation} \label{eq:g}
  \dot{f} = g_1 f - f g_0 =
  F_{\lambda_1} \left( F_{\lambda_1}^{-1}g_1 F_{\lambda_1} -
   F_{\lambda_0}^{-1}g_0 F_{\lambda_0} \right) F_{\lambda_0}^{-1}.
\end{equation}
Setting  $B_0 = F_{\lambda_0}^{-1}g_0 F_{\lambda_0},\,
B_1 = F_{\lambda_1}^{-1}g_1 F_{\lambda_1}$, then equation
\eqref{eq:g} reads
\begin{equation*}
  F_{\lambda_1}^{-1}\dot{f} F_{\lambda_0} = B_1 - B_0  =
  \tau \,\left(\alpha'_{\lambda_1}-\alpha'_{\lambda_0}\right)
  + \sigma \,
  \left(\alpha''_{\lambda_1} -\alpha''_{\lambda_0}\right)
    + \omega\,\epsilon\,.
\end{equation*}
The derivatives together with equations \eqref{eq:fdot} yield
\begin{equation*} \begin{split}
    F_{\lambda_1}^{-1}\partial \dot{f} F_{\lambda_0} &=
    B_1 \,\left(\alpha'_{\lambda_1} -\alpha'_{\lambda_0}\right)
  - \left(\alpha'_{\lambda_1}-\alpha'_{\lambda_0}\right)\,B_0 \\
    &= -\tfrac{1}{2}\sigma v^2 \mathbbm{1} +
    (\partial \tau + 2\tau\partial u - \omega H)
    \left(\alpha'_{\lambda_1}-\alpha'_{\lambda_0}\right) +
      (\tfrac{1}{2}\sigma v^2 H - \tau Q + \partial \omega)
      \,\epsilon\,, \\
  F_{\lambda_1}^{-1}\bar{\partial}\dot{f} F_{\lambda_0} &=
  B_1 \,\left(\alpha''_{\lambda_1}-\alpha''_{\lambda_0}\right)
  - \left(\alpha'_{\lambda_1}-\alpha'_{\lambda_0}\right)\,B_0 \\
  &= -\tfrac{1}{2}\sigma v^2 \mathbbm{1} +
    (\bar\partial \sigma + 2\sigma \bar{\partial} u - \omega H)
    \left(\alpha''_{\lambda_1}-\alpha''_{\lambda_0}\right) +
      (\tfrac{1}{2}\tau v^2 H - \sigma \overline{Q} +
      \bar{\partial} \omega)
      \,\epsilon\,.
\end{split}
\end{equation*}
These equations can be solved and we obtain
% B_0= \frac{1}{2}\begin{pmatrix}
%    \partial\tau + 2\tau\partial u - \omega(i+H) &
%    \tfrac{\overline{Q}^{-1}}{e^u}
%    (i\sigma\overline{Q}-\tfrac{1}{2}
%    \tau v^2(1+iH) - i\bar{\partial}\omega) \\
%    \tfrac{Q^{-1}}{e^u}(i\tau Q + \tfrac{1}{2}
%    \tau v^2(1-iH) -i \partial \omega) &
%    -\partial\tau - 2\tau\partial u + \omega(i+H)
%  \end{pmatrix} \\
% B_1= \frac{1}{2}\begin{pmatrix}
%    \partial\tau + 2\tau\partial u + \omega(i-H) &
%    \tfrac{\overline{Q}^{-1}}{e^u}
%    (i\sigma\overline{Q}+\tfrac{1}{2}
%    \tau v^2(1-iH) - i\bar{\partial}\omega) \\
%    \tfrac{Q^{-1}}{e^u}(i\tau Q - \tfrac{1}{2}
%    \tau v^2(1+iH) -i \partial \omega) &
%    -\partial\tau - 2\tau\partial u - \omega(i-H)
%  \end{pmatrix} \\
%
\begin{equation*} \begin{split}
  B_0 &= \frac{1}{2i} \left(\partial\tau +
    2\tau\partial u - \omega(i+H) \right) \,\epsilon
    + v^{-2}(i\sigma\overline{Q}-\tfrac{1}{2}
    \tau v^2(1+iH) - i\bar{\partial}\omega)
    \,\left(\alpha'_{\lambda_1}-\alpha'_{\lambda_0}\right)\\
    &\qquad - v^{-2}(i\tau Q + \tfrac{1}{2}
    \sigma v^2(1-iH) -i \partial \omega)
    \,\left(\alpha''_{\lambda_1}-\alpha''_{\lambda_0}\right)\,,
    \end{split}
\end{equation*}
\begin{equation*} \begin{split}
  B_1 &= \tfrac{1}{2i}\left(
  \partial\tau + 2\tau\partial u + \omega(i-H) \right)
  \,\epsilon + v^{-2}(i\sigma\overline{Q}+\tfrac{1}{2}
    \tau v^2(1-iH) - i\bar{\partial}\omega)
    \,\left(\alpha'_{\lambda_1}-\alpha'_{\lambda_0}\right)
     \\
    &\qquad -v^{-2}(i\tau Q - \tfrac{1}{2}
    \sigma v^2(1+iH) -i \partial \omega)
    \,\left(\alpha''_{\lambda_1}-\alpha''_{\lambda_0}\right)\,.
        \end{split}
\end{equation*}
Consequently, setting
$g_j = F_{\lambda_j} B_j F^{-1}_{\lambda_j}$ for $j=0,\,1$,
we get
\begin{equation*} \begin{split}
  f\,g_0 &= \frac{1}{2i} \left(\partial\tau +
    2\tau\partial u - \omega(i+H) \right) \,N
    + v^{-2}(i\sigma\overline{Q}-\tfrac{1}{2}
    \tau v^2(1+iH) - i\bar{\partial}\omega)
    \,\partial f \\
    &\qquad - v^{-2}(i\tau Q + \tfrac{1}{2}
    \sigma v^2(1-iH) -i \partial \omega)
    \,\bar{\partial}f\,, \\
  g_1\,f &= \tfrac{1}{2i}\left(
  \partial\tau + 2\tau\partial u + \omega(i-H) \right)
  \,N + v^{-2}(i\sigma\overline{Q}+\tfrac{1}{2}
    \tau v^2(1-iH) - i\bar{\partial}\omega)
    \,\partial f \\
        &\qquad -v^{-2}(i\tau Q - \tfrac{1}{2}
    \sigma v^2(1+iH) -i \partial \omega)
    \,\bar{\partial}f.
\end{split}
\end{equation*}
It remains to prove that $\dot{u} = 0$ implies that $dg_1 = dg_2
=0$, or equivalently that both $d(fg_0) = df f^{-1} fg_0$ and $d(g_1
f) = g_1 f f^{-1}df$ hold. To this end it is useful to recall the
identities
\begin{equation*} \begin{split}
  \partial f \,f^{-1} \partial f = 0\,, &\qquad
  \bar{\partial}f\,f^{-1}\bar{\partial}f =0\,,\\
  \bar{\partial}f \,f^{-1} \partial f = -\tfrac{1}{2}v^2(f+iN)\,,
  &\qquad
  \partial f \,f^{-1} \bar{\partial}f = -\tfrac{1}{2}v^2(f-iN)\,,\\
  N f^{-1} \partial f = -\partial f f^{-1} N = i\partial f\,,
  &\qquad
  -N f^{-1} \bar{\partial} f =  \bar{\partial} f f^{-1} N =
  i \bar{\partial} f\,.
  \end{split}
\end{equation*}
The rest of the proof is now a straightforward computation.
\end{proof}
%
%%%%%%%%%%%%%%%%%%%%%%%%%%%%%%%%%%%%%%%%%%%%%%%%%%%%%%%%%%%%%%%
%%%%%%%%%%%%%%%%%%%%%%%%%%%%%%%%%%%%%%%%%%%%%%%%%%%%%%%%%%%%%%%

\section{The homogeneous Jacobi fields}
We next calculate a solution of the homogeneous Jacobi equation
\eqref{eq:inhom_jacobi}. Let $u:\R^2 \to \R$ be a solution of the
sinh-Gordon equation. Let $\psi = (\psi_1,\,\psi_2)^t$ be a solution
of
\begin{equation} \label{eq:psi} \begin{split}
  \partial \begin{pmatrix} \psi_1 \\ \psi_2 \end{pmatrix} &=
  \frac{1}{2}\,\begin{pmatrix}
  \partial u & ie^{-u} \\ i\lambda^{-1}e^u & -\partial u
  \end{pmatrix}
  \begin{pmatrix} \psi_1 \\ \psi_2 \end{pmatrix}\,, \\
  \bar{\partial} \begin{pmatrix} \psi_1 \\ \psi_2
  \end{pmatrix} &=
  \frac{1}{2}\,\begin{pmatrix}
  -\bar{\partial}u & i\lambda e^u \\ ie^{-u} &
  \bar{\partial}u \end{pmatrix}
  \begin{pmatrix} \psi_1 \\ \psi_2 \end{pmatrix}.
\end{split}
\end{equation}
Then $\varphi = (\varphi_1,\,\varphi_2)^t$ defined by
\begin{equation*}
  \begin{pmatrix} \varphi_1 \\ \varphi_2 \end{pmatrix} =
  \begin{pmatrix} 0 & 1 \\ -1 & 0 \end{pmatrix}
  \begin{pmatrix} \psi_1 \\ \psi_2 \end{pmatrix}
\end{equation*}
solves
\begin{equation} \label{eq:phi} \begin{split}
  \partial \begin{pmatrix} \varphi_1 \\ \varphi_2 \end{pmatrix} &=
  -\frac{1}{2}\,\begin{pmatrix}
  \partial u & i\lambda^{-1} e^u \\ ie^{-u} & -\partial u
  \end{pmatrix}
  \begin{pmatrix} \varphi_1 \\ \varphi_2 \end{pmatrix}\,, \\
  \bar{\partial} \begin{pmatrix} \varphi_1 \\ \varphi_2
  \end{pmatrix} &=
  -\frac{1}{2}\,\begin{pmatrix}
  -\bar{\partial}u & ie^{-u} \\ i\lambda e^u &
  \bar{\partial}u \end{pmatrix}
  \begin{pmatrix} \varphi_1 \\ \varphi_2 \end{pmatrix}\,.
\end{split}
\end{equation}
For every value of $\lambda\in\C^*$ there exists a two dimensional
space of such functions $\psi$ and $\varphi$. Assume $u$ is periodic
with period $\gamma$. Then we assume in addition to \eqref{eq:psi}
and \eqref{eq:phi} that these functions are eigenvalues of the
translations by the period $\gamma$. Hence we get for any pair
$(\lambda,\mu)$, such that $\mu\not=\pm 1$ is an eigenvalue of the
monodromy of $\psi$, a one dimensional space space of such functions
$\psi$. A normalization of this function extends to a unique
function $\psi$ on the spectral curve called Baker-Akhiezer
function. Since the monodromy of $\psi$ has determinant one, there
always exist another solution $\sigma^*\psi$ with inverse eigenvalue
of the monodromy. We shall assume that the translations by the
periods acts on $\varphi$ with inverse eigenvalues as on $\psi$.
This of course can be only achieved with $\varphi$ obtained from
another solution $\sigma^*\psi$ of \eqref{eq:psi}. So we may assume
that both $\psi_1\varphi_1$ and $\psi_2\varphi_2$ have trivial
multipliers with respect to $\gamma$.
\begin{proposition} \label{thm:omega_tau_sigma}
Let $\psi = (\psi_1,\,\psi_2)^t$ be a solution of \eqref{eq:psi} and
let $\varphi = (\varphi_1,\,\varphi_2)^t$ be a solution of
\eqref{eq:phi}. Define $\omega = \psi_1 \varphi_1 - \psi_2\varphi_2$
and  $\phi = \partial \omega - 2Q\tau$. Then

\noindent {\rm{(i)}} The function $y = \psi^t\varphi$ satisfies
$dy = 0$.

\noindent {\rm{(ii)}} The function $\omega$ is in the kernel of the
Jacobi operator, and can be supplemented to a parametric Jacobi
field with corresponding (up to complex constants)
\begin{equation} \label{eq:tau&sigma}
  \tau = \frac{i\psi_2\varphi_1}{e^u Q} \,, \quad
  \sigma = -\frac{i\psi_1\varphi_2}{e^u \overline{Q}}\,.
\end{equation}
\noindent {\rm{(iii)}}
For the parametric Jacobi field generated by $\omega,\,\tau$
and $\sigma$, the variation of the conformal factor is
\begin{equation} \label{eq:dotu}
  \dot{u} = \frac{(\lambda-\lambda_0)(\lambda-\lambda_1)}
      {\lambda\,(\lambda_0 - \lambda_1)}\,i\,\omega\,.
\end{equation}

\noindent {\rm{(iv)}} Then $\left(\partial\omega\right)^2 - \phi^2 =
  \lambda^{-1}\left( y^2-\omega^2 \right)$.

\noindent {\rm{(v)}} Further,
    $\partial \phi = 2\,\partial u\,\partial \omega$, and
    $\lambda^{-1}\omega = 2\phi\,\partial u - \partial^2\omega$.
\end{proposition}
\begin{proof}
(i) is a straightforward computation using \eqref{eq:psi} and
\eqref{eq:phi}. To prove (ii), we compute
\begin{equation} \begin{split} \label{eq:omega_derivatives}
  \partial \omega = ie^{-u}\psi_2\varphi_1 -
  i\lambda^{-1}e^u \psi_1\varphi_2&\,,\\
   \bar{\partial} \partial \omega = -\omega\,\cosh(2u)&\,,\\
  \bar{\partial} \left( ie^{-u}\psi_2\varphi_1 +
  i\lambda^{-1}e^u \psi_1\varphi_2 \right) = \omega\,\sinh(2u)&\,.
  \end{split}
\end{equation}
Hence $\omega$ is in the kernel of the Jacobi operator.
From the $\bar{\partial}$-derivative
\eqref{eq:sigma_z} of $\tau$ and
\begin{align*}
  \bar{\partial}(-2Q\tau + \partial \omega) &=
  -2Q\bar{\partial}\tau
  + \bar{\partial}\partial\,\omega = \omega\,e^{-2u} -
  \omega\,\cosh(2u) \\
  &= -\omega\,\sinh(2u) =
  -\bar{\partial}\left(ie^{-u}\psi_2\varphi_1
  +i\lambda^{-1}e^u\psi_1\varphi_2 \right),
\end{align*}
we see that up to a complex constant
$-2Q\tau + \partial \omega = ie^{-u}\psi_2\varphi_1
  +i\lambda^{-1}e^u\psi_1\varphi_2$, and consequently
obtain the formula for $\tau$.
A similar computation yields the formula
for $\sigma$. To prove (iii) we compute
\begin{equation*}
   \tfrac{1}{2}\,\partial\tau + \tau \,\partial u
   = - \tfrac{\lambda^{-1}}{4Q} \,\omega\,, \quad
   \tfrac{1}{2}\,\bar{\partial}\sigma + \sigma \,\bar{\partial} u
   = - \tfrac{\lambda}{4\overline{Q}} \,\omega\,.
\end{equation*}
From \eqref{eq:udot} we have $\dot{u} =
  \tfrac{1}{2}\partial \tau + \tau \partial u +
  \frac{1}{2}\bar{\partial} \sigma +
    \sigma \bar{\partial}u - \omega H$.
A straightforward computation using \eqref{eq:H_mu} and \eqref{eq:Q}
now proves \eqref{eq:dotu}.

From equations \eqref{eq:psi} and \eqref{eq:phi} and
\eqref{eq:tau&sigma} we obtain
\begin{equation}\label{eq:domega&phi}\begin{split}
  \partial\omega &= ie^{-u}\varphi_1\psi_2 -
  i\lambda^{-1}e^{u}\varphi_2\psi_1
  = Q\,\tau + \lambda^{-1}e^{2u}\overline{Q}\,\sigma\,,\\
  \phi &= \partial\omega-2Q\,\tau = -Q\,\tau +
  \lambda^{-1}e^{2u}\overline{Q}\,\sigma.
\end{split}\end{equation}
Consequently $\tau\sigma$ is equal to
\begin{equation*}
  \tau\sigma=\frac{1}{2Q}\left(\partial\omega-\phi\right)
  \frac{\lambda}{2e^{2u}\overline{Q}}
  \left(\partial\omega+\phi\right)=\lambda v^{-2}
\left(\left(\partial\omega\right)^2-\phi^2\right).
\end{equation*}
On the other hand equations \eqref{eq:tau&sigma} yield
$$
\tau\,\sigma = \frac{e^{-2u}}{Q\overline{Q}}
\psi_1\psi_2\varphi_1\varphi_2 = \frac{e^{-2u}}{4Q\overline{Q}}
\left(y^2-\omega^2\right) = v^{-2}\left(y^2-\omega^2\right)\,,
$$
and equating these proves (iv). To prove (v), we compute $\partial
\phi = \partial^2\omega-2Q\partial\tau =
\partial^2 \omega
- \partial^2 \omega + 2\,\partial u\,\partial \omega =
2\,\partial u\,\partial \omega$, and
$(-\lambda^{-1})\omega = 2Q\,\partial \tau + 4Q\,\tau \partial u
= \partial^2\omega - \partial\phi + 2\partial\omega\,\partial u -
2\phi\,\partial u = \partial^2\omega - 2\phi\,\partial u$.
\end{proof}
%
%%%%%%%%%%
%
\subsection{Involutions}
In the sequel we shall consider solutions $u$ of
\eqref{eq:sinh-Gordon} corresponding to a spectral curve with
involutions $\sigma$, $\eta$ and $\varrho$ and not necessarily
normalized Baker-Akhiezer function $\psi$. The maps $\psi =
(\psi_1,\,\psi_2)^t$ as in \eqref{eq:psi} and $\varphi =
(\varphi_1,\,\varphi_2)^t$ as in \eqref{eq:phi} transform as follows
under the involutions of the spectral curve:
\begin{align*}
  \begin{pmatrix} \varphi_1 \\ \varphi_2 \end{pmatrix}
  &\sim \begin{pmatrix} 0 & 1 \\ -1 & 0 \end{pmatrix}
  \sigma^* \begin{pmatrix} \psi_1 \\ \psi_2 \end{pmatrix}\,, \qquad
  \begin{pmatrix} \psi_1 \\ \psi_2 \end{pmatrix}
  \sim \begin{pmatrix} 0 & 1 \\ -1 & 0 \end{pmatrix}
  \sigma^* \begin{pmatrix} \varphi_1 \\ \varphi_2 \end{pmatrix}\,, \\
  \begin{pmatrix} \psi_1 \\ \psi_2 \end{pmatrix}
  &\sim \begin{pmatrix} 0 & 1 \\ -1 & 0 \end{pmatrix}
  \eta^* \begin{pmatrix} \overline{\psi_1} \\
    \overline{\psi_2} \end{pmatrix} \,,\qquad
  \begin{pmatrix} \varphi_1 \\ \varphi_2 \end{pmatrix}
  \sim \begin{pmatrix} 0 & 1 \\ -1 & 0 \end{pmatrix}
  \eta^* \begin{pmatrix} \overline{\varphi_1} \\
    \overline{\varphi_2} \end{pmatrix} \,,\\
  \begin{pmatrix} \varphi_1 \\ \varphi_2 \end{pmatrix}
  &\sim \begin{pmatrix} 0 & 1 \\ -1 & 0 \end{pmatrix}
  \varrho^* \begin{pmatrix} \overline{\psi_1} \\
    \overline{\psi_2} \end{pmatrix}\,.
\end{align*}
If $u$ is periodic with respect to at least one period $\gamma$,
then the eigenvalues of the corresponding monodromy define the
spectral curve with these involutions $\sigma$, $\eta$ and
$\varrho$. In this case the solutions $\psi$ of \eqref{eq:psi} and
$\varphi$ of \eqref{eq:phi} diagonalize the translations by the
periods with inverse eigenvalues $\mu$ and $\mu^{-1}$. Hence up to
normalization $\psi$ is the Baker-Akhiezer function.
\begin{proposition}
Let $u$ be a solution of \eqref{eq:sinh-Gordon} of finite type and
$\psi = (\psi_1,\,\psi_2)^t$ respectively $\varphi =
(\varphi_1,\,\varphi_2)^t$ be the corresponding solutions of
\eqref{eq:psi} and \eqref{eq:phi}. We define
\begin{equation} \label{eq:P}
  P = \frac{\psi \,\varphi^t}{\psi^t \,\varphi}\,.
\end{equation}
Then
\begin{equation} \begin{split} \label{eq:delu_res}
  \partial u &= \frac{i}{2}\,\mathrm{res}_{\lambda=0}
  \left( \lambda^{-1}d\lambda^{1/2} \,\mathrm{tr} \left(
  \begin{pmatrix} 1&0\\0&-1\end{pmatrix}\,P \right) \right)\,, \\
  \bar{\partial} u &= -\frac{i}{2}\,\mathrm{res}_{\lambda=\infty}
  \left( \lambda \,d\lambda^{-1/2}\, \mathrm{tr} \left(
  \begin{pmatrix} 1&0\\0&-1\end{pmatrix}\,P \right) \right)\,.
    \end{split}
\end{equation}
\end{proposition}
\begin{proof}
Define $\widetilde{\psi}_1,\,\widetilde{\psi}_2$ and
$\widetilde{\varphi}_1,\,\widetilde{\varphi}_2$ by the gauge
\begin{equation} \begin{split} \label{eq:psitilde}
  \begin{pmatrix}\widetilde{\psi}_1 \\
    \widetilde{\psi}_2 \end{pmatrix} &=
  \begin{pmatrix} \lambda^{-1/2} e^{u/2}&0\\
    0&e^{-u/2}\end{pmatrix}
  \begin{pmatrix}\psi_1 \\ \psi_2 \end{pmatrix} \\
  \begin{pmatrix}\widetilde{\varphi}_1 \\
    \widetilde{\varphi}_2 \end{pmatrix} &=
  \begin{pmatrix} \lambda^{1/2} e^{-u/2}&0\\0& e^{u/2}\end{pmatrix}
  \begin{pmatrix}\varphi_1 \\ \varphi_2 \end{pmatrix} \\
\end{split}
\end{equation}
Then $\widetilde{\psi}_1,\,\widetilde{\psi}_2$ and
$\widetilde{\varphi}_1,\,\widetilde{\varphi}_2$ are solutions
of
\begin{align*} %\label{eq:psitilde_DE}
  \partial \begin{pmatrix} \widetilde{\psi}_1 \\
    \widetilde{\psi}_2 \end{pmatrix} &=
  \frac{1}{2}\,\begin{pmatrix}
  2\partial u & i\lambda^{-1/2} \\
  i\lambda^{-1/2} & -2\partial u
  \end{pmatrix}
  \begin{pmatrix} \widetilde{\psi}_1 \\
    \widetilde{\psi}_2 \end{pmatrix}\,, \quad
  \bar{\partial} \begin{pmatrix} \widetilde{\psi}_1 \\
    \widetilde{\psi}_2
  \end{pmatrix} =
  \frac{1}{2}\,\begin{pmatrix}
   0 & i\lambda^{1/2} e^{2u} \\
   i\lambda^{1/2} e^{-2u} & 0 \end{pmatrix}
  \begin{pmatrix} \widetilde{\psi}_1 \\
    \widetilde{\psi}_2 \end{pmatrix}\,, \\
%%
%\label{eq:phitilde_DE}
  \partial \begin{pmatrix} \widetilde{\varphi}_1 \\
    \widetilde{\varphi}_2 \end{pmatrix} &=
  -\frac{1}{2}\,\begin{pmatrix}
  2\partial u & i\lambda^{-1/2}  \\
  i\lambda^{-1/2} & -2\partial u \end{pmatrix}
  \begin{pmatrix} \widetilde{\varphi}_1 \\
    \widetilde{\varphi}_2 \end{pmatrix}\,, \quad
  \bar{\partial} \begin{pmatrix} \widetilde{\varphi}_1 \\
    \widetilde{\varphi}_2
  \end{pmatrix} =
  -\frac{1}{2}\,\begin{pmatrix}
    0 & i\lambda^{1/2} e^{-2u} \\
    i\lambda^{1/2} e^{2u} & 0 \end{pmatrix}
  \begin{pmatrix} \widetilde{\varphi}_1 \\
    \widetilde{\varphi}_2 \end{pmatrix}\,.
\end{align*}
Hence the asymptotic expansions at $\lambda = 0$ are
\begin{equation*} \begin{split} %\label{eq:asymp_tilde}
  \begin{pmatrix} \widetilde{\psi}_1 \\
    \widetilde{\psi}_2 \end{pmatrix} &=
  \exp\left(\tfrac{iz}{2}\lambda^{-1/2}\right) \left(
    \begin{pmatrix}1\\1\end{pmatrix} -
      i \lambda^{1/2} \begin{pmatrix}\partial u\\
    -\partial u\end{pmatrix}
      + O(\lambda) \right)\,, \\
  \begin{pmatrix} \widetilde{\varphi}_1 \\
    \widetilde{\varphi}_2 \end{pmatrix} &=
  \exp\left(-\tfrac{iz}{2} \lambda^{-1/2}\right) \left(
    \begin{pmatrix}1\\1\end{pmatrix} -
      i\lambda^{1/2}\begin{pmatrix}\partial u\\
    -\partial u\end{pmatrix}
      + O(\lambda) \right)\,.
\end{split}
\end{equation*}
Then
\begin{equation*}
  \widetilde{P} = \frac{\widetilde{\psi}\,\widetilde{\varphi}^t}
        {\widetilde{\psi}^t\,\widetilde{\varphi}} =
        \frac{1}{2}\begin{pmatrix}1&1\\1&1\end{pmatrix} -
  i \lambda^{1/2}\begin{pmatrix}\partial u&0\\
    0&-\partial u\end{pmatrix} +
  O(\lambda)
\end{equation*}
and consequently
\begin{equation}\label{eq:asymP} \begin{split}
  P &= \begin{pmatrix} \lambda^{1/2}e^{-u/2}&0\\
    0&e^{u/2}\end{pmatrix} \,
  \widetilde{P}\,
  \begin{pmatrix} \lambda^{-1/2}e^{u/2}&0\\
    0&e^{-u/2}\end{pmatrix} \\
  &= \frac{1}{2}\begin{pmatrix}1&\lambda^{1/2} e^{-u} \\
        \lambda^{-1/2}e^u&1\end{pmatrix} -
  i\lambda^{1/2}\begin{pmatrix}\partial u&0\\
    0&-\partial u\end{pmatrix} +
  O(\lambda),
  \end{split}
\end{equation}
which implies the first equation in \eqref{eq:delu_res}.
Analogous computations at $\lambda = \infty$ prove
the second equation in \eqref{eq:delu_res}.
\end{proof}
Recall the iteration in Pinkall--Sterling \cite{PinS} that
generates a sequence of solutions to the homogeneous Jacobi
equation: Starting with the 'trivial' solution
$\omega_1 = \partial u$, then
$\varphi_1 = (\partial u)^2$ and thus with the slightly
different normalization in our setting we obtain
$\tau_1 = \partial \omega_1 - \varphi_1$,
which yields a second solution
$$
\omega_2 = \partial^3 u - 2 (\partial u)^3
$$
of the homogeneous Jacobi equation.
This iterative procedure at each step requires for a given
$\omega_n$ to find $\tau_n$ solving both
\begin{equation*} \begin{split}
  \partial \tau_n &= \partial^2 \omega_n
    - 2\partial\omega_n\partial u \,,\\
  \bar{\partial} \tau_n &=
  -4Q\overline{Q}v^{-2}\omega_n =-e^{-2u}\omega_n\,,
  \end{split}
\end{equation*}
and then defining $\omega_{n+1} = \partial\tau_n + 2\tau_n\partial
u$. To solve these equations for $\tau_n$, it is useful to introduce
the auxiliary functions $\phi_n$ such that
\begin{equation} \label{eq:tau_n}
  \tau_n  = \partial\omega_n - \phi_n .
\end{equation}
Then $\phi_n$ satisfies
\begin{equation*} \begin{split}
  \partial \phi_n &= 2\partial\omega_n\partial u\,, \\
  \bar{\partial} \phi_n &= -\omega_n \sinh(2u)\,.
\end{split}
\end{equation*}
To supplement $\omega_n$ and $\tau_n$ at each step
to a parametric Jacobi field requires finding a function
$\sigma_n$ satisfying
\begin{equation*} \begin{split}
  \partial \sigma_n &=  -4Q\overline{Q}v^{-2}\omega_n\,, \\
  \bar{\partial} \sigma_n &= \bar{\partial}^2 \omega_n
    - 2\bar{\partial}\omega_n\bar{\partial} u \,.
  \end{split}
\end{equation*}
In analogy to Lemma 3.2
in Pinkall and Sterling \cite{PinS}, we obtain the formula
\begin{equation}\label{eq:sigma_n}
  \sigma_n = -e^{-2u}\left( \partial \omega_{n-1}
    + \phi_{n-1} \right)\,.
\end{equation}
We shall now see, that all these solutions
fit together in the Taylor
expansions of the functions $\omega,\,\tau,\,\phi$ at $\lambda=0$:

Let
\begin{equation*}
  \Phi = \sum_{n=0}^\infty (-1)^n \phi_n \lambda^n\,,\quad
  \Omega = \sum_{n=0}^\infty (-1)^n \omega_n \lambda^n\,.
\end{equation*}
\begin{proposition}
Let $y = \psi^t \varphi$. Then the entries of
\begin{equation*}
  P = \frac{\psi \varphi^t}{\psi^t \varphi} = \frac{1}{y}
  \begin{pmatrix} \psi_1 \varphi_1 & \psi_1 \varphi_2 \\
    \psi_2 \varphi_1 & \psi_2 \varphi_2 \end{pmatrix}
\end{equation*}
have at $\lambda=0$ the asymptotic expansions
\begin{equation} \begin{split}\label{eq:series}
  -\frac{i\omega}{2y} &= \frac{1}{\sqrt{\lambda}}\,
  \sum_{n=1}^\infty\omega_n (-\lambda)^n\,, \\
  -\frac{iQ\tau}{y}&=e^{-u}\psi_2 \varphi_1 = \frac{1}{\sqrt{\lambda}}\,
  \sum_{n=0}^\infty \tau_n (-\lambda)^n\,, \\
  -\frac{i\overline{Q}\sigma}{y}&=-e^{-u}\psi_1 \varphi_2 =
  \frac{1}{\sqrt{\lambda}}\,
  \sum_{n=1}^\infty \sigma_n (-\lambda)^n\,,\\
  -\frac{i\phi}{2y}&=\frac{iQ\tau-i\lambda^{-1}
    e^{2u}\overline{Q}\sigma}{2y}=
  \frac{-e^{-u}\psi_2\varphi_1-\lambda^{-1}e^u\psi_1\varphi_2}{2y}=
  \frac{1}{\sqrt{\lambda}}\sum_{n=0}^\infty \phi_n(-\lambda)^n\,.
\end{split}
\end{equation}
\end{proposition}
\begin{proof}
Due to \eqref{eq:domega&phi} the coefficients of the series in
\eqref{eq:series} obey the equations \eqref{eq:tau_n} and
\eqref{eq:sigma_n}. Hence it suffices to show that these Taylor
coefficients obey the recursion formula
\cite[Proposition~3.1]{PinS}. The first and last equations in
\eqref{eq:series} read $\omega = 2yi\lambda^{-1/2}\Omega$ and $\phi
= 2yi\lambda^{-1/2}\Phi$. From Proposition \ref{thm:omega_tau_sigma}
(iv) we have $\left(\partial\omega\right)^2 - \phi^2 =
\lambda^{-1}\left( y^2-\omega^2 \right)$ and thus $\Phi^2 -
(\partial \Omega)^2 = \tfrac{1}{4} + \lambda^{-1} \Omega^2$, or in
terms of series
\begin{equation*}
  \left(\sum_{n=0}^\infty (-1)^n\phi_n \lambda^n\right)^2-
  \left(\sum_{n=1}^\infty (-1)^n \partial\omega_n\lambda^n\right)^2
  =
  \frac{1}{4} + \lambda^{-1}
  \left(\sum_{n=1}^\infty (-1)^n \omega_n \lambda^n\right)^2.
\end{equation*}
Due to \eqref{eq:asymP} we conclude $\phi_0=-\frac{1}{2}$.
Consequently, for all $n\in\mathbb{N}_0$ the coefficient
$\phi_{n+1}$ is equal to the coefficient of $(-\lambda)^{n+1}$ of
the series
\begin{equation*}
  \left(\sum_{n=1}^\infty (-1)^n \phi_n \lambda^n\right)^2-
  \left(\sum_{n=1}^\infty (-1)^n\partial\omega_n \lambda^n\right)^2
  - \lambda^{-1}
  \left(\sum_{n=1}^\infty (-1)^n \omega_n \lambda^n\right)^2.
\end{equation*}
\end{proof}
The involution $\varrho$ allows us to compute the asymptotic
expansions at $\lambda = \infty$ from the the asymptotic expansions
at $\lambda = 0$. Note that $\bar{\tau} = \sigma$ and thus
$\varrho^*(\tfrac{\bar{\omega}}{y}) = \tfrac{\omega}{y}$ and
$\varrho^*(\tfrac{\bar{\tau}}{y}) = \tfrac{\sigma}{y}$. We summarize
the asymptotic expansions at $\lambda = \infty$ in the following
\begin{corollary}
At $\lambda=\infty$ we have the asymptotic expansions
\begin{equation} \begin{split}\label{eq:series2}
  \frac{i\omega}{2y} &= \lambda^{1/2}\,
  \sum_{n=1}^\infty (-1)^n \overline{\omega_n}\,\lambda^{-n}\,, \\
  \frac{iQ\tau}{y}&= \lambda^{1/2}\,
  \sum_{n=1}^\infty (-1)^n \overline{\sigma_n}\,\lambda^{-n}\,, \\
  \frac{i\overline{Q}\sigma}{y} &= \lambda^{1/2}\,
  \sum_{n=0}^\infty (-1)^n \overline{\tau_n}\,\lambda^{-n} \,.
\end{split}
\end{equation}
\end{corollary}
Utilizing the fact that
\begin{equation}
  \frac{\omega}{y} = \mathrm{tr} \left(
  \bigl( \begin{smallmatrix} 1&0\\0&-1\end{smallmatrix} \bigr)
    \,P \right)\,,
\end{equation}
we obtain the following
\begin{corollary}
We have the asymptotic expansions
\begin{equation*} \begin{split}
  \mathrm{tr} \left(
  \bigl( \begin{smallmatrix} 1&0\\0&-1\end{smallmatrix} \bigr)
    \,P \right) &=
    2i\lambda^{-1/2}
    \left( -\lambda\,\partial u + \lambda^2 (\partial^3 u -
    2(\partial u)^3 ) + \mathrm{O}(\lambda^3) \right)
    \quad \mbox{ at }
    \lambda = 0\,, \\
  \mathrm{tr} \left(
  \bigl( \begin{smallmatrix} 1&0\\0&-1\end{smallmatrix} \bigr)
    \,P \right) &=
    -2i\lambda^{1/2}\left( -\lambda^{-1} \bar{\partial} u +
    \lambda^{-2} (\bar{\partial}^3 u -
    2(\bar{\partial} u)^3 ) + \mathrm{O}(\lambda^{-3}) \right)
    \quad \mbox{ at }
    \lambda = \infty\,.
    \end{split}
\end{equation*}
\end{corollary}
%
%%%%%%%%%%%%%%%%%%%%%%%%%%%%%%%%
%
The poles of $P$ are at the branch points of the spectral curve.
Clearly the function $y = \psi^t\varphi$ is antisymmetric with
respect to the hyperelliptic involution, and has zeroes at the
branch points $a_1,\,\ldots,\,a_{2g}$. For doubly periodic solutions
$u$ of \eqref{eq:sinh-Gordon} the differentials $d\ln
\mu_k,\,k=1,\,2$ can be locally written as
\begin{equation*}
  d\ln\mu_k = \beta_k(y)\,dy\,,
\end{equation*}
and if $\gamma_1$ and $\gamma_2$ denote the periods, there exist
two differentials $dp^{\pm}$ which are holomorphic at
$\lambda = \infty$ respectively $\lambda = 0$ and
\begin{equation} \label{eq:dp^pm}
   d\ln\mu_k = \gamma_k\,dp^+ +\bar{\gamma}_k\,dp^- \,.
\end{equation}
If $\Gamma$ is the lattice generated by $\gamma_1$ and $\gamma_2$,
$\mathbb{T} = \C / \Gamma$, and $f:\mathbb{T} \to \C$, then
we denote
\begin{equation*}
  \langle f \rangle = \frac{1}{\mathrm{area}(\mathbb{T})}
  \int_{\mathbb{T}} f\,dS\,.
\end{equation*}
For a period $\gamma$ let $\Delta_\gamma$ denote the period defect,
which is the difference between the identity operator and the
translation by $\gamma$.
\begin{proposition} \label{th:dp^pm}
The differentials $dp^\pm$ have the following asymptotic
expansions
\begin{xalignat*}{2}
  dp^+ &= d\lambda^{1/2} \left( -\tfrac{i}{2}\lambda^{-1} -
  i\langle (\partial u)^2 \rangle + \mathrm{O}(\lambda) \right)
  &\mbox{ at } \lambda = 0\,,\\
  dp^- &= d\lambda^{1/2} \left( \tfrac{i}{2}
  \langle \cosh(2u) \rangle + \mathrm{O}(\lambda) \right)
  &\mbox{ at } \lambda = 0\,,\\
  dp^+ &= d\lambda^{-1/2} \left(
  \tfrac{i}{2}\langle \cosh(2u) \rangle
  + \mathrm{O}(\lambda^{-1} \right)
  &\mbox{ at } \lambda = \infty\,,\\
  dp^- &= d\lambda^{-1/2} \left( -\tfrac{i}{2}\lambda -
  i\langle (\partial u)^2 \rangle + \mathrm{O}(\lambda^{-1} \right)
  &\mbox{ at } \lambda = \infty\,.
\end{xalignat*}
\end{proposition}
\begin{proof}
Let $\widetilde\psi = (\widetilde{\psi}_1,\,\widetilde{\psi}_2)$
as in \eqref{eq:psitilde} and define $\hat{\psi}$ by
$\widetilde\psi = \exp(p^+ z + p^- \bar{z})\, \hat{\psi}$.
Since $\mu_k$ is the automorphy factor of $\widetilde{\psi}$
with respect to $\gamma_k$, then
$\Delta_{\gamma_k}\hat\psi = 0$, and we expand
\begin{equation*}
        \begin{split}
  \hat{\psi}(z,\,\lambda) &= \begin{pmatrix}
    1\\1\end{pmatrix} \,\left(\sum a_n(z) \lambda^{n/2} \right)
    + \begin{pmatrix}
    1\\-1\end{pmatrix} \,\left(\sum  b_n(z)\lambda^{n/2}\right)
    \quad \mbox{ with }a_0=1,\,b_0=0\,, \\
  p^+(\lambda) &= \sum p^+_n\,\lambda^{n/2}\,,\quad
    p^-(\lambda) = \sum p^-_n\,\lambda^{n/2} \quad \mbox{ and } p_{-1}^+ =
    \tfrac{i}{2},\, p_{-1}^- = 0\,.
    \end{split}
\end{equation*}
The differential equations for $\widetilde{\psi}$
and comparison of like coefficients yields
\begin{equation*} \begin{split}
  \partial a_n &= b_n \partial u - p_1^+ a_{n-1} -
  \ldots - p_n^+ a_0\,, \\
  i\,b_{n+1} &= a_n \partial u - p_1^+b_{n-1} -
  \ldots - p^+_{n-1}b_1 -
  \partial b_n\,, \\
  \bar{\partial} a_n &= - p_1^- a_{n-1} - \ldots - p_n^- a_0 +
    \tfrac{i}{2}(a_{n-1}\cosh(2u) - b_{n-1}\sinh(2u))\,, \\
    \bar{\partial} b_n &= - p_1^- b_{n-1} - p_3^- b_{n-3} - \ldots
    - p_{n-1}^- b_1 +
    \tfrac{i}{2}(a_{n-1}\sinh(2u) - b_{n-1}\cosh(2u))\,.
    \end{split}
\end{equation*}
In particular, $b_1 = -i\partial u$. Further
$\partial a_1 = -i(\partial u)^2 - p_1^+$, so integration yields
$p_1^+ = -i\langle (\partial u)^2 \rangle$.
Thus  $\partial a_1 = -i(\partial u)^2 -
i\langle (\partial u)^2 \rangle$ which gives
$$
  \bar{\partial}\partial a_1 = i\partial u \sinh(2u) =
  \partial \tfrac{i}{2} \cosh(2u)
$$
and therefore $\bar{\partial}a_1 = -p^-_1 + \tfrac{i}{2} \cosh(2u)$
and consequently $p_1^- = \langle \tfrac{i}{2} \cosh(2u) \rangle$.
This proves the asymptotic expansions at $\lambda = 0$.
The asymptotic expansions at $\lambda = \infty$ follow from the
fact that
$$
  \varrho^* \overline{d \ln \mu} = -d \ln \mu,
$$
and therefore $\varrho^* \overline{dp^+} = -dp^-$
and $\varrho^* \overline{dp^-} = -dp^+$.
\end{proof}
We now compute the variation of the Baker-Akhiezer function $\psi$
that corresponds to the solution $\dot{u}= \left. (\psi_1 \varphi_1
- \psi_2 \varphi_2) \right|_{\lambda=a}$ of the homogeneous Jacobi
equation at some fixed value $\lambda=a \in \C^*$.
\begin{proposition}\label{th:delta_u}
Let $u$ be a solution of the sinh-Gordon equation
\eqref{eq:sinh-Gordon}. Let $\xi_{ij} = \left. (\psi_i \varphi_j)
\right|_a$ and $\dot{u} = \xi_{11}-\xi_{22}$ be given in terms of
solutions $\psi = (\psi_1,\,\psi_2)^t$ of \eqref{eq:psi} and
$\varphi = (\varphi_1,\,\varphi_2)^t$ of \eqref{eq:phi} for a fixed
value $\lambda=a\in\C^*$. Then the variation of the corresponding
solution $\psi$ of $\eqref{eq:psi}$ is given by
\begin{equation} \label{eq:psi_dot}
  \begin{pmatrix} \dot{\psi_1} \\
    \dot{\psi_2} \end{pmatrix} =
  \frac{1}{\lambda - a} \begin{pmatrix}
    (\lambda + a)\,\xi_{11} & 2\lambda\,\xi_{12} \\
    2a\,\xi_{21} & (\lambda + a)\,\xi_{22} \end{pmatrix}
  \begin{pmatrix} \psi_1 \\
    \psi_2 \end{pmatrix}\,.
\end{equation}
\end{proposition}
\begin{proof}
By \eqref{eq:psi} we have that
\begin{equation*}  \begin{split}
  \partial \begin{pmatrix} \dot{\psi_1} \\
    \dot{\psi_2} \end{pmatrix} &=
  \frac{1}{2}\,\begin{pmatrix}
  \partial u & ie^{-u} \\ i\lambda^{-1}e^u & -\partial u
  \end{pmatrix}
  \begin{pmatrix} \dot{\psi_1} \\ \dot{\psi_2} \end{pmatrix}
   + \frac{1}{2}\,\begin{pmatrix}
  \partial \dot{u} & -i\dot{u}e^{-u} \\
  i\dot{u}\lambda^{-1}e^u & -\partial \dot{u} \end{pmatrix}
  \begin{pmatrix} \psi_1 \\ \psi_2 \end{pmatrix}\,, \\
  %%%
  \bar{\partial} \begin{pmatrix} \dot{\psi_1} \\
    \dot{\psi_2} \end{pmatrix} &=
  \frac{1}{2}\,\begin{pmatrix}
  -\bar{\partial}u & i\lambda e^u \\ ie^{-u} &
  \bar{\partial}u \end{pmatrix}
  \begin{pmatrix} \dot{\psi_1} \\ \dot{\psi_2} \end{pmatrix}
  + \frac{1}{2}\,\begin{pmatrix}
  -\bar{\partial}\dot{u} & i\dot{u}\lambda e^u \\
  -i\dot{u}e^{-u} & \bar{\partial}\dot{u} \end{pmatrix}
  \begin{pmatrix} \psi_1 \\ \psi_2 \end{pmatrix}.
\end{split}
\end{equation*}
Consequently,
\begin{equation} \label{eq:del_u_dot}
    \partial \dot{u} = ie^{-u}\xi_{21} -
    ia^{-1}e^u \xi_{12}\,,\quad
    \bar{\partial} \dot{u} = iae^{u}\xi_{21} -
    ie^{-u} \xi_{12}\,.
\end{equation}
Let $X = \bigl( \begin{smallmatrix} \alpha & \beta \\ \gamma &
\delta \end{smallmatrix} \bigr)$ be the matrix such
that $\dot{\psi} = X\,\psi$. Then
\begin{equation} \begin{split} \label{eq:A_equation}
  2\,\partial \begin{pmatrix} \alpha & \beta \\ \gamma &
    \delta \end{pmatrix} + \left[\,
    \begin{pmatrix} \alpha & \beta \\ \gamma &
      \delta \end{pmatrix},\, \begin{pmatrix}
  \partial u & ie^{-u} \\ i\lambda^{-1}e^u & -\partial u
  \end{pmatrix}\,\right]
  &=  \begin{pmatrix}
  \partial \dot{u} & -i\dot{u}e^{-u} \\
  i\dot{u}\lambda^{-1}e^u & -\partial\dot{u}\end{pmatrix}\,,\\
 2\,\bar{\partial} \begin{pmatrix} \alpha & \beta \\ \gamma &
    \delta \end{pmatrix} + \left[\,
    \begin{pmatrix} \alpha & \beta \\ \gamma &
      \delta \end{pmatrix},\, \begin{pmatrix}
  -\bar{\partial} u & i\lambda e^u \\ ie^{-u} & \bar{\partial}u
  \end{pmatrix}\,\right]
  &= \begin{pmatrix}
  -\bar{\partial}\dot{u} & i\dot{u}\lambda e^u \\
  -i\dot{u}e^{-u} & \bar{\partial}\dot{u} \end{pmatrix}\,.
\end{split}
\end{equation}
We compute the commutators
\begin{equation*} \begin{split}
  \left[\,
    \begin{pmatrix} \alpha & \beta \\ \gamma &
      \delta \end{pmatrix},\,\begin{pmatrix}
  \partial u & ie^{-u} \\ i\lambda^{-1}e^u & -\partial u
  \end{pmatrix}\,\right]
  &=  \begin{pmatrix} \beta i\lambda^{-1}e^u -
  \gamma i e^{-u} & (\alpha-\delta)ie^{-u} -2\beta\partial u\\
  2\gamma\partial u + (\delta - \alpha)i\lambda^{-1}e^u &
  \gamma ie^{-u} - \beta i \lambda^{-1} e^u \end{pmatrix}\,, \\
  \left[\,
    \begin{pmatrix} \alpha & \beta \\ \gamma &
      \delta \end{pmatrix},\, \begin{pmatrix}
  -\bar{\partial} u & i\lambda e^u \\ ie^{-u} & \bar{\partial}u
  \end{pmatrix}\,\right]
  &= \begin{pmatrix}
  \beta i e^{-u} - \gamma i \lambda e^u &
  (\alpha - \delta)i \lambda e^u + 2\beta \bar{\partial}u \\
  (\delta - \alpha)ie^{-u}-2\gamma\bar{\partial}u &
  \gamma i\lambda e^u - \beta ie^{-u} \end{pmatrix}\,.
  \end{split}
\end{equation*}
Thus equations \eqref{eq:A_equation} read
\begin{equation*} \begin{split}
  2\,\partial \begin{pmatrix} \alpha & \beta \\ \gamma &
    \delta \end{pmatrix}
  &= \begin{pmatrix}
  (\gamma +\xi_{21})ie^{-u} -
    (a^{-1}\xi_{12}+\beta\lambda^{-1})ie^u &
  (\xi_{22}-\xi_{11} -\alpha + \delta)ie^{-u} +
  2\beta\partial u \\
  (\xi_{11}-\xi_{22}+\alpha-\delta)i\lambda^{-1}e^u -
  2\gamma\partial u &
  (a^{-1}\xi_{12} + \beta \lambda^{-1})ie^u -
  (\xi_{21} + \gamma)ie^{-u} \end{pmatrix},\\
 2\,\bar{\partial} \begin{pmatrix} \alpha & \beta \\ \gamma &
    \delta \end{pmatrix}
  &=
 \begin{pmatrix}
  (\xi_{12}-\beta) i e^{-u} +(\gamma\lambda -a\xi_{21})ie^u &
  (\xi_{11}-\xi_{22}-\alpha + \delta)i \lambda e^u -
   2\beta \bar{\partial}u \\
  (\alpha-\delta+\xi_{22}-\xi_{11})ie^{-u} +
   2\gamma\bar{\partial}u &
   (\beta-\xi_{12})ie^{-u} + (a\xi_{21}-\gamma\lambda)i e^u
 \end{pmatrix}.
\end{split}
\end{equation*}
We make the Ansatz
$\alpha =A\,\xi_{11},\,\beta =B\,\xi_{12},\,\gamma =C\,\xi_{21}$
and $\delta = A\,\xi_{22}$, and use the fact that
\begin{equation*}\begin{split}
    2\,\partial \xi_{11}= \xi_{21}ie^{-u}- a^{-1}\xi_{12}ie^u\,,
    \quad
    &2\,\bar{\partial}\xi_{11}=-\xi_{12}ie^{-u}+a\xi_{21}ie^u\,,\\
    2\,\partial \xi_{12}=2\,\xi_{12}\,\partial u
    +(\xi_{22}-\xi_{11})ie^{-u}\,,
    \quad
    &2\,\bar{\partial}\xi_{12}=-2\,\xi_{12}\,\bar{\partial}u
    + a(\xi_{22}-\xi_{11})ie^u\,, \\
    2\,\partial \xi_{21}= -2\,\xi_{21}\,\partial u +
    a^{-1}(\xi_{11}-\xi_{22})ie^u\,,
    \quad
    &2\,\bar{\partial}\xi_{21}= 2\,\xi_{21}\,\bar{\partial}u +
    (\xi_{11}-\xi_{22})ie^{-u}\,, \\
    2\,\partial \xi_{22}=-\xi_{21}ie^{-u} + a^{-1}\xi_{12}ie^u\,,
    \quad
    &2\,\bar{\partial} \xi_{22} = \xi_{12}ie^{-u} -
    a^{-1}\xi_{21}ie^u\,.
\end{split}
\end{equation*}
Comparison of coefficients of $e^{-u},\,e^u$ reduce to
the three equations
\begin{equation*}
  A-C =1\,, \quad B\lambda^{-1} = a^{-1}(A-1)\mbox{ and }
  B=A+1\,.
\end{equation*}
Solving these yields the claim and concludes the proof.
\end{proof}
The Fermi curve contains a lot of information of the spectral
theory of the Jacobi operator. In particular, it should be possible to
determine the Maslov index in terms of the spectral curve of the
solution of the sinh-Gordon equation.
%
%%%%%%%%%%%%%%%%%%%%%%%%%%%%%%%%%%%%%%%%%%%%%%%%%%%%%%%%%%%%%
%
\begin{lemma}
Let $u$ be a doubly periodic solution of the sinh-Gordon equation,
and let $\psi = (\psi_1,\,\psi_2)^t$ be a solution of
\eqref{eq:psi}. Then $\psi_1\psi_2$ is the Baker-Akhiezer function
of the Fermi curve of the Jacobi operator.
\end{lemma}
\begin{proof}
A straightforward computation shows that if $\psi =
(\psi_1,\,\psi_2)^t$ solves \eqref{eq:psi}, then $\psi_1\psi_2$ is
in the kernel of the Jacobi operator. With respect to the two
periods $\gamma_1,\,\gamma_2$ of $u$ we have $\psi(z+\gamma_j) =
\mu_j \,\psi(z)$, and thus for $j=1,\,2$ we have
\begin{equation}
    (\psi_1\psi_2)\, (z+\gamma_j) = \mu_j^2\,(\psi_1\psi_2)\,(z).
\end{equation}
Hence the spectral curve is a finite covering of the Fermi curve
of the Jacobi operator. By Theorem 17.9 in \cite{FelKT:inf},
the kernel of the Jacobi operator is generically one dimensional,
and therefore the spectral
curve is a simple covering of the Fermi curve.
\end{proof}
Let us finally indicate, how we may construct inhomogenous Jacobi
fields. Denote the real quantity
$$
\mathfrak{H} = - \frac{\dot{H}}{2(H^2+1)} = \frac{\dot{Q}}{2HQ}.
$$
Let $\hat{\omega} = \omega^+ + \omega^-$ where
$\omega^+ = \mathfrak{H}\left(z\,\partial u  - \tfrac{1}{2}\right)$
and $\omega^- = \mathfrak{H}\left(\bar{z}\, \bar{\partial}u -
\tfrac{1}{2}\right)$. A straightforward computation shows that
$\hat{\omega}$ is also solution of the inhomogeneous Jacobi equation
\eqref{eq:inhom_jacobi}.
%
%\begin{equation} \label{eq:inhom_jacobi_half}
%  \bar{\partial}\partial\omega  + \cosh(2u)\,\omega =
%  -\frac{\mathfrak{H}}{2}\,e^{2u}\,.
%\end{equation}
%
The corresponding
$\hat{\tau} = \tau^+ + \tau^-$ and
$\hat{\sigma} = \sigma^+ + \sigma^-$ that supplement $\hat\omega$
to a parametric Jacobi field are given by
$\sigma^+ = \overline{\tau^-}$ and
$\sigma^- = \overline{\tau^+}$, where
\begin{equation*} \begin{split}
  \tau^+ &= \frac{\partial \omega^+}{2Q} -
  \frac{\mathfrak{H}\,z\,(\partial u)^2}{2Q}
  - \frac{\mathfrak{H}}{4Q} \int_0^z 2(\partial u)^2\,dw -
  \cosh(2u)\,d\bar{w}\,, \\
  \tau^- &= \mathfrak{H}Hz +
  \frac{\mathfrak{H}}{4Q}\,\bar{z}\,e^{-2u}\,.
  \end{split}
\end{equation*}
The period defects of these functions are given by
\begin{equation*} \begin{split}
  \Delta_{\gamma}\left( \tfrac{1}{2}\partial \tau^+ +
  \tau^+ \partial u
    \right) &= \frac{\mathfrak{H}\,\gamma}{4Q}
    \left( \partial^3 u -
    2(\partial u)^3 \right) - \frac{\mathfrak{H}\,\partial u}{4Q}
    \int_z^{z+\gamma} \hspace{-3mm}2(\partial u)^2\,dw -
  \cosh(2u)\,d\bar{w}\,,\\
  \Delta_{\gamma}\left( \tfrac{1}{2}\partial \tau^-
  + \tau^- \partial u
    \right) &= \mathfrak{H}\,H \,\gamma \,\partial u\,,\qquad
  \Delta_{\gamma}\left( \tfrac{1}{2}\bar{\partial} \sigma^+ +
  \sigma^+ \bar{\partial} u
    \right) = \mathfrak{H}\,H \,\bar{\gamma} \,\bar{\partial}u\,,\\
  \Delta_{\gamma}\left( \tfrac{1}{2}\bar{\partial} \sigma^-
    + \sigma^- \bar{\partial} u
    \right) &= \frac{\mathfrak{H}\,\bar{\gamma}}{4\overline{Q}}
    \left( \bar{\partial}^3 u -
    2(\bar{\partial} u)^3 \right) +
      \frac{\mathfrak{H}\,\bar{\partial} u}{4\overline{Q}}
  \int_z^{z+\gamma}\hspace{-3mm}2(\bar{\partial} u)^2\,d\bar{w} -
  \cosh(2u)\,dw\,.
    \end{split}
\end{equation*}
We may add to these inhomogenous Jacobi fields homogenous Jacobi
fields with the same period defects and obtain periodic inhomogenous
Jacobi fields.
%%%%%%%%%%%%%%%%%%%%%%%%%%%%%%%%%%%%%%%%%%%%%%%%%%%%%
%%%%%%%%%%%%%%%%%%%%%%%%%%%%%%%%%%%%%%%%%%%%%%%%%%%%%
%%%%%%%%%%%%%%%%%%%%%%%%%%%%%%%%%%%%%%%%%%%%%%%%%%%%%

\section{Polynomial Killing fields}

Let $\upsilon = \bigl( \begin{smallmatrix} 0 & 1 \\ 0 & 0
  \end{smallmatrix} \bigr)$, and denote Laurent
polynomials $\C^\times \to \Sl$ with a simple pole at
$\lambda = 0$ and normalized leading coefficients, by
\begin{equation*}
  \Lambda_g = \bigl\{ \xi_0:\C^\times \to \Sl \,\left|\, \right.
  \xi_0(\lambda) = \sum_{j=-1}^g c_j \lambda^j\,,
  \lambda^{g-1}\overline{\xi(1/\bar{\lambda})}^t = -\xi(\lambda)
  \mbox{ and } c_{-1} = \upsilon \bigr\}\,.
\end{equation*}
The condition $\lambda^{g-1}\overline{\xi(1/\bar{\lambda})}^t =
-\xi(\lambda)$ is a reality condition which ensures that
the shifted polynomial $\lambda^{-l}\,\xi$ takes values in
$\su$ on $\mathbb{S}^1$, where
\begin{equation*}
  l = \left\{ \begin{tabular}{ll}
    $\frac{1}{2}(g+1)$ & if $g$ is odd, \\
    $\tfrac{1}{2}(g-1)$ & if $g$ is even. \end{tabular} \right.
\end{equation*}
We denote these skew-hermitian loops by
\begin{equation*}
  \Lambda_g \su = \left\{ \lambda^{-l}\,\xi_0 \left| \, \right.
  \xi_0 \in \Lambda_g \right\}\,.
\end{equation*}
By the Symes method \cite{Symes_80}, the extended framing
$F:\R^2 \times \C^\times \to \SU$ of any harmonic map
$\R^2 \to \mathbb{S}^2$ of finite type
(see Burstall and Pedit \cite{BurP:dre,BurP_adl})
is given by the unitary factor of the Iwasawa decomposition
of
\begin{equation} \label{eq:FB}
  \exp(z\,\xi_0) = F\,B
\end{equation}
for some $\xi_0 \in \Lambda_g$. The loop $\xi_0 \in \Lambda_g$ is
called a \emph{potential} of the corresponding harmonic map in the
generalized Weierstrass representation of Dorfmeister, Pedit and Wu
\cite{DorPW}. By a result of Pinkall and Sterling \cite{PinS}, and
independently Hitchin \cite{Hit:tor}, all Gauss maps of {\sc{cmc}}
tori are of finite type, so equivalently one may solve first solve
the sinh-Gordon equation to obtain some function $u:\R^2 \to \R$ and
then solve $dF = F\,\alpha$ with $F(0) = \mathbbm{1}$ with $\alpha$
as in \eqref{eq:general_alpha} to obtain an extended framing. Recall
that a \emph{polynomial Killing field} in this case is a map $\xi :
\R^2 \to \Lambda_g \su$ which solves
\begin{equation} \label{eq:Killing}
  d\xi = [\,\xi,\,\alpha\,],\,
  \quad \xi(0) = \lambda^{-l} \xi_0
\end{equation}
with $\xi_0 \in \Lambda_g$. The solution to \eqref{eq:Killing}
via the Iwasawa decomposition \eqref{eq:FB} is then
$\xi = F^{-1} \lambda^{-l} \xi_0 \,F =
B\,\lambda^{-l} \xi_0 \, B^{-1}$.
\begin{proposition}
Let $\xi$ be a polynomial Killing field and $\psi_0$ the eigenvector
and $\varphi_0$ the transposed eigenvector of $\xi(0)$:
\begin{equation*}
\xi(0)\,\psi = \nu\,\psi\,,\quad \varphi^t\,\xi(0)= \nu \,\varphi^t
\quad \mbox{ with } \quad \nu^2 =-\det(\xi(0))= \det(\xi).
\end{equation*}
Then the eigenvector $\psi$ and transposed eigenvector $\varphi^t$ of
$\xi$ are the solutions of
\begin{equation*}
d\psi = -\alpha\,\psi \,\mbox{ with } \,\psi(0) = \psi_0\,; \quad
d\varphi^t = \varphi^t\,\alpha \, \mbox{ with } \,
\varphi^t(0)=\varphi_0^t.
\end{equation*}
\end{proposition}
\begin{proof}
The unique solutions are given by $\varphi^t = \varphi_0^t\,F$ and
$\psi = F^{-1}\,\psi_0$, and are eigenvectors of the
polynomial Killing field.
\end{proof}
Conversely, assume we are given $\alpha$ as in
\eqref{eq:general_alpha}
with some periodic finite gap solution of the
sinh-Gordon equation.
Then $\psi,\,\varphi^t$ are the corresponding
Baker-Akhiezer functions if
\begin{equation} \begin{split} \label{eq:BA-functions}
  &d\varphi^t = \varphi^t\,\alpha \mbox{ and } \varphi^t M_F =
    \nu\, \varphi^t\,, \\
  &d\psi = -\alpha\,\psi  \mbox{ and } M_F \psi =
    \nu\, \psi\,,
\end{split}
\end{equation}
where $M_F$ is the monodromy of the extended framing with respect to
the period. Note that $\psi$ and $\varphi^t$ extend meromorphically
into points where $M_F$ is not semi-simple. Furthermore, there
exists a unique polynomial Killing field $\xi$ such that
\begin{equation}
  \varphi^t \xi = \nu\,\varphi^t \mbox{ and }
  \xi\,\psi = \mu\,\psi\,.
\end{equation}
We call $\xi$ a Killing field for $\psi$ and $\varphi$. If
$\varphi^t(0)= \varphi_0^t$ and $\psi(0)= \psi_0^t$, then $Q =
\psi\,\varphi^t$ solves $dQ = [Q,\,\alpha]$ with $Q(0)
=\psi_0\,\varphi_0^t$. Note that $Q$ has rank $\leq 1$ everywhere.
At the zeroes of $\det \xi$ (the branch points of the spectral
curve), we shall describe the dynamics of the spectral curve under
the isoperiodic deformations in terms of the Baker-Akhiezer
functions $\psi,\,\varphi^t$.

Let $u$ be a solution of the sinh-Gordon equation, and $\psi =
(\psi_1,\,\psi_2)^t$ and $\varphi = (\varphi_1,\,\varphi_2)^t$ be
solutions of \eqref{eq:psi} and \eqref{eq:phi} respectively. Let
$\xi_{ij} = \left. (\psi_i \varphi_j) \right|_a$ and $\dot{u} =
\xi_{11}-\xi_{22}$. Recall from Proposition \ref{th:delta_u} that
\begin{equation}\label{eq:pre_dot_psi}
  \dot{\psi}  =
  \frac{1}{\lambda - a} \begin{pmatrix}
    (\lambda + a)\,\xi_{11} & 2\lambda\,\xi_{12} \\
    2a\,\xi_{21} & (\lambda + a)\,\xi_{22} \end{pmatrix}
  \psi \,.
\end{equation}
Let $y$ be an arbitrary point on the spectral curve,
$a=\lambda(y)$ and $Q = \psi(y)\varphi(y)^t$. Then an
easy computation shows that the previous equation
\eqref{eq:pre_dot_psi} can be rewritten as
\begin{equation} \label{eq:dot_psi}
  \dot{\psi} = \tfrac{1}{\lambda - a} \left(
  \bigl( \begin{smallmatrix} \lambda & 0\\0 & a
    \end{smallmatrix} \bigr)\,Q +
  Q \,\bigl( \begin{smallmatrix} a & 0\\0 & \lambda
    \end{smallmatrix} \bigr) \right)\,\psi\,.
\end{equation}
\begin{lemma} The transformation $\dot{\psi}$
is an infinitesimal isospectral transformation, that is,
$\dot{u} = \left.(\psi_1\varphi_1-\psi_2\varphi_2)\right|_a$
is isospectral. All isospectral transformations are
obtained by taking linear combinations of such transformations at all
branch points $a$. The isospectral transformations are a
$g$-dimensional space, where $g$ is the arithmetic genus of the
spectral curve.
\end{lemma}
\begin{proof}
Let $\xi$ be a polynomial Killing field for $\psi$ and $\varphi$,
and $a \in \C^*$ any branch point of the spectral curve. Then
$\xi(a)\,Q = \nu\,Q = Q\,\xi(a)$.

Since $Q$, and the matrices $\bigl( \begin{smallmatrix} \lambda & 0\\0 & a
    \end{smallmatrix} \bigr)$ and $\bigl( \begin{smallmatrix} a & 0\\0 & \lambda
    \end{smallmatrix} \bigr)$ commute with $\xi$ at $\lambda=a$, the
commutator $\left[ \left(
  \bigl( \begin{smallmatrix} \lambda & 0\\0 & a
    \end{smallmatrix} \bigr)\,Q +
  Q \,\bigl( \begin{smallmatrix} a & 0\\0 & \lambda
    \end{smallmatrix} \bigr) \right),\,\xi \right]$
has a zero at $\lambda =a$.

From \eqref{eq:dot_psi} and $\dot{\xi}\,\psi + \xi\,\dot{\psi} =
\nu\,\dot{\psi} + \dot{\nu}\,\psi = \nu\,\dot{\psi}$, we conclude that
$$
\dot{\xi} = \tfrac{1}{\lambda - a}\,\left[ \left(
  \bigl( \begin{smallmatrix} \lambda & 0\\0 & a
    \end{smallmatrix} \bigr)\,Q +
  Q \,\bigl( \begin{smallmatrix} a & 0\\0 & \lambda
    \end{smallmatrix} \bigr) \right),\,\xi \right]
$$
is a tangent vector in the complexified space of Killing fields at the point
$\xi$ with $\dot{\nu}=0$.

Now we claim that all such transformations are linear combinations
of such transformations at all branch points of the spectral curve.
The meromorphic map $P = \tfrac{\psi\,\varphi^t}{\psi^t\,\varphi}$
has poles only at the branch points. For $g=0$, the space of Killing
fields is zero dimensional, and there are no infinitesimal
isospectral transformations. For $g \geq 1$ there exists for each
point a meromorphic differential $\omega$ with a first order pole at
that point, and first order zeroes at $\lambda =0$ and $\lambda
=\infty$. The sum over all residues of $\omega\,P$ vanishes, and
consequently, the value of $P$ at such a point is a linear
combination of the values of $P$ at the branch points. This proves
the claim.

Finally we remark, that for all holomorphic one-forms $\omega$, the
sum over all residues of $\omega P$ vanishes. Therefore all
non-trivial holomorhic one-forms yield a non-trivial relation on the
corresponding infinitesimal transformations. Therefore the space of
these transformations span a $g$-dimensional space of isospectral
transformations.
\end{proof}
In the next lemma we exhibit the generators for which
$\dot{u} = \left.\partial(\psi_1\psi_2)\right|_a$
is non-isospectral.
\begin{lemma}
Let $u$ be a solution of the sinh-Gordon equation, and
$\psi = (\psi_1,\,\psi_2)^t$ and
$\varphi = (\varphi_1,\,\varphi_2)^t$ be solutions of
\eqref{eq:psi} and \eqref{eq:phi} respectively.
Let
\begin{equation}
  Q = \tfrac{d}{dy} \left.\left( \psi \,\,\sigma^*\varphi^t
  \right) \right|_{\lambda=a}
\end{equation}
be the derivative with respect to a local parameter $y$ at some
branch point $\lambda=a \in \Sigma_g$. Then
\begin{equation}
  \dot{\psi} = \tfrac{1}{\lambda - a} \left(
  \bigl( \begin{smallmatrix} \lambda & 0\\0 & a
    \end{smallmatrix} \bigr)\,Q +
  Q \,\bigl( \begin{smallmatrix} a & 0\\0 & \lambda
    \end{smallmatrix} \bigr) \right)\,\psi
\end{equation}
is an infinitesimal non-isospectral transformation that only moves
the branch point $a$ while keeping all the other branch points
fixed. The space of all these deformations generates the space of
all non-isospectral transformations.
\end{lemma}
\begin{proof}
Since the branch points are zeroes of $d\lambda$, the derivative of
$\lambda$ with respect to $y$ vanishes. Hence we may apply
Proposition~\ref{th:delta_u} in this situation. Since the square of
the eigenvalue $\nu$ of $\xi$ is equal to $\lambda^{-1}$ times a
polynomial with respect to $\lambda$, whose roots are the branch
points in $\C^*$, we have to show that $2\nu\dot{\nu}$ vanishes at
all branch points in $\C^*$ with the exception of $\lambda=a$. Hence
$\dot{\nu}$ has to be proportional to $\det(\xi)/(\nu(\lambda-a))$.
From \eqref{eq:dot_psi} and $\dot{\xi}\,\psi + \xi\,\dot{\psi} =
\nu\,\dot{\psi} + \dot{\nu}\,\psi = \nu\,\dot{\psi}$ and
$\det(\xi)\xi^{-1}=-xi$, we conclude that
$$
\left[ \left(
  \bigl( \begin{smallmatrix} \lambda & 0\\0 & a
    \end{smallmatrix} \bigr)\,Q +
  Q \,\bigl( \begin{smallmatrix} a & 0\\0 & \lambda
    \end{smallmatrix} \bigr) \right),\,\xi \right]
$$
has to be proportional to $\xi$ at $\lambda=a$. If we differentiate
the equations
\begin{align*}
\xi\psi&=\nu\psi&\sigma^*\varphi^t\xi&=\nu\sigma^*\varphi^t
\end{align*}
at the branch points, we obtain that the commutator $[Q,\,\xi]$ is
at $\lambda=a$ proportional to $\xi$. This implies the claim.
\end{proof}
With the help of the inhomogenous Jacobi fields described at the end
of the last section, one may construct non-isospectral deformations
of {\sc{cmc}} tori in $\mathbb{S}^3$.
%
%%%%%%%%%%%%%%%%%%%%%%%%%%%%%%%%%%%%%%%%%%%%%%%%%%%%%
%%%%%%%%%%%%%%%%%%%%%%%%%%%%%%%%%%%%%%%%%%%%%%%%%%%%%
%%%%%%%%%%%%%%%%%%%%%%%%%%%%%%%%%%%%%%%%%%%%%%%%%%%%%

%%%%%%%%%%%%%%%%%%%%%%%%%%%%%%%%%%%%%%%%%%%%%%%%%%%%%
%%%%%%%%%%%%%%%%%%%%%%%%%%%%%%%%%%%%%%%%%%%%%%%%%%%%%
%%%%%%%%%%%%%%%%%%%%%%%%%%%%%%%%%%%%%%%%%%%%%%%%%%%%%

%%%%%%%%%%%%%%%%%%%%%%%%%%%%%%%%%%%%%%%%%%%%%%%%%%%%%
\bibliographystyle{amsplain}

\def\cprime{$'$}
\providecommand{\bysame}{\leavevmode\hbox
to3em{\hrulefill}\thinspace}
\providecommand{\MR}{\relax\ifhmode\unskip\space\fi MR }
% \MRhref is called by the amsart/book/proc definition of \MR.
\providecommand{\MRhref}[2]{%
  \href{http://www.ams.org/mathscinet-getitem?mr=#1}{#2}
} \providecommand{\href}[2]{#2}

%\bibliography{ref}

\end{document}